\documentclass[twocolumn]{article}

\usepackage{abstract} 

\usepackage{amsmath} 
\usepackage{amssymb}
\usepackage{amsthm}

\newtheorem{theorem}[equation]{Theorem}
\newtheorem{proposition}[equation]{Proposition}

\begin{document}

\title{Effective conductivity of a singularly perturbed periodic two-phase composite with imperfect thermal contact at the two-phase interface}

\author{Matteo Dalla Riva \and Paolo Musolino}

\date{}

\twocolumn[ 

\maketitle

\begin{onecolabstract} 
We consider the asymptotic behaviour of the effective thermal conductivity of a two-phase composite obtained by introducing into an infinite homogeneous matrix a periodic set of inclusions of a different material and of size proportional to a positive parameter $\epsilon$. We are interested in the case of imperfect  thermal contact at the two-phase interface.  Under suitable assumptions, we show that the effective thermal conductivity can be continued real analytically in the parameter $\epsilon$ around the degenerate value $\epsilon=0$, in correspondence of which the inclusions collapse to points.  The results presented here are obtained by means of an approach based on functional analysis and potential theory and are also part of a forthcoming paper by the authors.
\end{onecolabstract}

\vspace{11pt}



\vspace{11pt}

]


\section{Introduction}
This note is devoted to the analysis of the effective thermal conductivity of a two-phase periodic composite, consisting of a matrix and of a periodic set of inclusions, with thermal resistance at the two-phase interface. Two possibly different materials fill the matrix and the inclusions. We assume that these materials are homogeneous and isotropic heat conductors. As a consequence, the conductivity of each of these two materials is represented by a positive scalar. Moreover,  we assume that the size of each inclusion  is proportional to a certain parameter $\epsilon>0$, and that as $\epsilon$ tends to zero each inclusion collapses to a point. The normal component of the heat flux is assumed to be continuous at the composite interface, while we impose that the temperature field displays a jump proportional to the normal heat flux by means of a parameter $\rho(\epsilon)>0$. Such a discontinuity in the temperature field has been  largely investigated since 1941, when Kapitza carried out the first systematic study of thermal interface behaviour in liquid helium (see, {\it e.g.},  Swartz and Pohl \cite{SwPo89}, Lipton \cite{Li98} and references therein). In this note, we investigate the asymptotic behaviour of the effective thermal conductivity when the positive parameter $\epsilon$ is close to the degenerate value $0$. Benveniste and Miloh in \cite{BeMi86} introduced the expression which defines the effective conductivity of a composite with imperfect contact conditions by generalizing the dual theory of the effective behaviour of composites with perfect contact (see also Benveniste \cite{Be86} and for a review of the subject, {\it e.g.}, Drygas and Mityushev \cite{DrMi09}). By the argument of Benveniste and Miloh, in order to evaluate the effective conductivity, one has to study the thermal distribution of the composite when so called ``homogeneous conditions'' are prescribed.  
  As a consequence, we introduce a particular transmission problem with non-ideal contact conditions where we impose that the temperature field displays a fixed jump along a prescribed direction and is periodic in all the other directions (cf.~problem \eqref{bvpe} below). 

We fix once for all 
\[
n\in {\mathbb{N}}\setminus\{0,1 \}\,,\qquad \alpha \in ]0,1[\,.
\]
Then we introduce the periodicity cell $Q$ by setting
\[
Q\equiv ]0,1[^n\,.
\]
 We fix a bounded open connected subset $\Omega$ of $\mathbb{R}^{n}$ of  the Schauder class  $C^{1,\alpha}$ such that the complementary set of its closure $\mathrm{cl}\Omega$ is connected and that the origin $0$ of $\mathbb{R}^n$ belongs to $\Omega$. We note that, by requiring that $\mathbb{R}^n \setminus \mathrm{cl}\Omega$ is connected, we assume that the set $\Omega$ does not have holes. The set $\Omega$ represents the ``shape'' of the inclusions. 

Next we fix a point $p$ in the fundamental cell $Q$ and for each $\epsilon \in \mathbb{R}$, we set
\[
\Omega_{p,\epsilon}\equiv p+\epsilon\Omega \,.
\]

Clearly, there exists $\epsilon_{0}>0$ small enough, such that
\[
{\mathrm{cl}}\Omega_{p,\epsilon} \subseteq Q
\qquad\forall \epsilon\in]-\epsilon_{0},\epsilon_{0}[\,.
\]
For $\epsilon \in ]0,\epsilon_0[$, the set $\Omega_{p,\epsilon}$ represents the inclusion in the fundamental cell $Q$. We note that for $\epsilon=0$ the set $\Omega_{p,\epsilon}$ degenerates into the set $\{p\}$.

We are now in the position to define the periodic domains $\mathbb{S}[\epsilon]$ and $\mathbb{S}[\epsilon]^-$ by setting
\begin{align}
&{\mathbb{S}}[\epsilon]\equiv\bigcup_{z\in {\mathbb{Z}}^{n}}\left(
z+\Omega_{p,\epsilon}\right)= \mathbb{Z}^n + \Omega_{p,\epsilon}
\,,\nonumber \\
&{\mathbb{S}}[\epsilon]^{-}\equiv
{\mathbb{R}}^{n}\setminus{\mathrm{cl}}{\mathbb{S}} [\epsilon]\,,\nonumber 
\end{align}
for all $\epsilon\in ]-\epsilon_{0},\epsilon_{0}[ $. We observe that for $\epsilon=0$ the sets $\mathbb{S}[\epsilon]$ and $\mathbb{S}[\epsilon]^-$ degenerate into $p+\mathbb{Z}^n$ and into $\mathbb{R}^n \setminus (p+\mathbb{Z}^n)$, respectively.

Next, we take two positive constants $\lambda^+$, $\lambda^-$ and a function $\rho$ from $]0,\epsilon_{0}[$ to $]0,+\infty[$. For each $j \in \{1,\dots,n\}$ and $\epsilon \in ]0,\epsilon_0[$ we consider the following transmission problem for a pair of functions $(u^+_j,u^-_j)\in C^{1}(\mathrm{cl}\mathbb{S}[\epsilon])\times C^{1}(\mathrm{cl}\mathbb{S}[\epsilon]^{-})$:
\begin{equation}
\label{bvpe}
\left\{
\begin{array}{ll}
\Delta u^+_j=0 & {\mathrm{in}}\ {\mathbb{S}}[\epsilon]\,,\\
\Delta u^-_j=0 & {\mathrm{in}}\ {\mathbb{S}}[\epsilon]^{-}\,,
\\
u^+_j(x+e_h)=u^+_j(x)+\delta_{h,j} & \forall x \in \mathrm{cl}\mathbb{S}[\epsilon]\, ,\\
& \forall h \in \{1,\dots,n\}\,,\\
u^-_j(x+e_h)=u^-_j(x)+\delta_{h,j} & \forall x \in \mathrm{cl}\mathbb{S}[\epsilon]^{-}\, ,\\
&  \forall h \in \{1,\dots,n\} \,,\\
\lambda^-\frac{\partial u^-_j}{\partial\nu_{ \Omega_{p,\epsilon} }}(x)
=
\lambda^+\frac{\partial u^+_j}{\partial\nu_{ \Omega_{p,\epsilon} }}(x)& \forall x\in  \partial\Omega_{p,\epsilon}\,,\\
\lambda^+\frac{\partial u^+_j}{\partial\nu_{ \Omega_{p,\epsilon} }}(x)&\\\qquad=\frac{1}{\rho(\epsilon)}\bigl(u^-_j(x)-u^+_j(x)\bigr)& \forall x\in  \partial\Omega_{p,\epsilon}\,,\\
\int_{\partial \Omega_{p,\epsilon}}u^+_j(x)\, d\sigma_x=0\, ,
\end{array}
\right.
\end{equation}
where $\nu_{ \Omega_{p,\epsilon}}$ denotes the outward unit normal to $ \partial
\Omega_{p,\epsilon}$, and $\delta_{h,j}=1$ if $h=j$, $\delta_{h,j}=0$ if 
$h\neq j$ for all $h,j \in \{1,\dots,n\}$. Here $\{e_{1}$,\dots, $e_{n}\}$ denotes the canonical basis of ${\mathbb{R}}^{n}$. 

The functions $u^+_j$ and $u^-_j$ represent the temperature field in the inclusions occupying ${\mathbb{S}}[\epsilon]$ and in the matrix occupying ${\mathbb{S}}[\epsilon]^{-}$, respectively.  The parameters $\lambda^+$ and $\lambda^-$ represent the thermal conductivity of the materials which fill  the inclusions and the matrix, respectively, whereas the parameter $\rho(\epsilon)$ plays the role of the interfacial thermal resistivity. The fifth condition in \eqref{bvpe} means that the normal heat flux is continuous across the two-phase interface.  The sixth condition says that the temperature field has a jump proportional to the  normal heat flux by means of the parameter $\rho(\epsilon)$. The third and fourth conditions in \eqref{bvpe}  imply that the temperature distributions  $u^+_j$ and $u^-_j$ have a jump equal to $1$ in the direction $e_j$ and are periodic in all the other directions. Finally, the seventh condition in \eqref{bvpe} is an auxiliary condition which we introduce  in order to have uniqueness for the solution of problem \eqref{bvpe}. Since the effective conductivity is invariant for constant modifications of the temperature field, such a condition does not interfere in its definition. 

Boundary value problems of this type have been largely investigated in connection with the theory of heat conduction in two-phase periodic composites with imperfect contact conditions. Here we mention, \textit{e.g.}, Castro and Pesetskaya \cite{CaPe10}, Castro, Pesetskaya, and Rogosin \cite{CaPeRo09},  Drygas and Mityushev \cite{DrMi09}, Lipton \cite{Li98}, Mityushev \cite{Mi01}.

Boundary value problem \eqref{bvpe} is clearly singular for $\epsilon=0$. Indeed, both the domains $\mathbb{S}[\epsilon]$ and $\mathbb{S}[\epsilon]^-$ are degenerate when $\epsilon=0$. Moreover, the presence of the factor $\frac{1}{\rho(\epsilon)}$ may produce a further singularity if $\rho(\epsilon)\to 0$ as $\epsilon$ tends to $0^+$. In this note, we consider the case in which the limit
\[
r_{\ast}\equiv\lim_{\epsilon\to 0^+}\frac{ \epsilon}{\rho(\epsilon)}
\]
exists  finite in  $\mathbb{R}$. We emphasize that we make no regularity assumption on the function $\rho$. 

As is well knwon, for each $\epsilon \in ]0,\epsilon_0[$,  problem \eqref{bvpe} has a unique solution in $C^{1}(\mathrm{cl}\mathbb{S}[\epsilon])\times C^{1}(\mathrm{cl}\mathbb{S}[\epsilon]^{-})$. We denote such a solution by $(u^+_j[\epsilon], u^-_j[\epsilon])$. Then we introduce the effective conductivity matrix $\lambda^{\mathrm{eff}}[\epsilon]$ with $(k,j)$-entry  $\lambda^{\mathrm{eff}}_{kj}[\epsilon]$ defined by
\[
\begin{split}
\lambda^{\mathrm{eff}}_{kj}[\epsilon]&\equiv \lambda^+\int_{\Omega_{p,\epsilon}}\frac{\partial u^+_j[\epsilon](x)}{\partial x_k} \, dx\\
& + \lambda^-\int_{Q\setminus \mathrm{cl}\Omega_{p,\epsilon}}\frac{\partial u^-_j[\epsilon](x)}{\partial x_k} \, dx \,,
\end{split}
\]
for all $(k,j) \in \{1,\dots,n\}^2$ and $\epsilon \in ]0,\epsilon_0[$ (cf.~Benveniste \cite{Be86} and Benveniste and Miloh \cite{BeMi86}). 

Then if $(k,j) \in \{1,\dots,n\}^2$ it is natural to ask the following question.
\begin{equation}\label{question}
\begin{split}
&\text{What  can be said on the map 
$\epsilon\mapsto \lambda^{\mathrm{eff}}_{kj} [\epsilon]$} \\
&\text{when $\epsilon$ is close to $0$ and 
positive?} 
\end{split}
\end{equation}

Questions of this type are not new and have long been investigated with the methods of Asymptotic Analysis. 

Thus for example,  one could resort to the techniques of Asymptotic Analysis and may succeed  to write out an asymptotic  expansion for $\lambda^{\mathrm{eff}}_{kj} [\epsilon]$ of the type
\[
\lambda^{\mathrm{eff}}_{kj} [\epsilon]=P(\epsilon)+R(\epsilon) \qquad \text{as $\epsilon \to 0^+$},
\]
where $P$ is a regular function and $R$ a remainder which is smaller than a positive known function of $\epsilon$.

Here, we mention, as an example, the works of Ammari and Kang \cite[Ch.~5]{AmKa07}, Ammari, Kang, and Touibi \cite{AmKaTo05}, Maz'ya, Nazarov, and Plamenewskij \cite{MaNaPl00i, MaNaPl00ii}, 
Maz'ya, Movchan, and Nieves \cite{MaMoNi11} (for further references see, {\it e.g.}, Lanza de Cristoforis and the second named author \cite{LaMu12}). 

In this note, instead, we wish to answer to the question in \eqref{question} by exploiting the different approach proposed by Lanza de Cristoforis. Namely, our aim is to represent $\lambda^{\mathrm{eff}}_{kj} [\epsilon]$ when $\epsilon$ is small and positive in terms of real analytic functions of the variable $\epsilon$ defined on a whole neighbourhood of $0$, and of  explicitly known functions of $\epsilon$. This approach does have its advantages. Indeed, if we know, for example, that there exist $\epsilon' \in ]0,\epsilon_0[$ and a real analytic function $h$ from $]-\epsilon',\epsilon'[$ to $\mathbb{R}$ such that
\[
\lambda^{\mathrm{eff}}_{kj} [\epsilon]=h(\epsilon) \qquad \forall \epsilon \in ]0,\epsilon'[\, ,
\]
then we can deduce the existence of $\epsilon'' \in ]0,\epsilon'[$ and of a sequence $\{a_j\}_{j=0}^{+\infty}$ of real numbers, such that
\[
\lambda^{\mathrm{eff}}_{kj} [\epsilon]=\sum_{j=0}^{+\infty}a_j \epsilon^j \qquad \forall \epsilon \in ]0,\epsilon''[\, ,
\]
where the series in the right hand side converges absolutely on $]-\epsilon'',\epsilon''[$. As we shall see, this is the case if $\epsilon/\rho(\epsilon)$ has a real analytic continuation around $0$ (for example if $\rho(\epsilon)=\epsilon$ or $\rho$ is constant).

Such a project has been carried out in the case of a simple hole, \textit{e.g.}, in Lanza \cite{La10} (see also \cite{DaMu12}), and has later been    extended to problems related to the system of equations of the linearized elasticity in \cite{DaLa10a, DaLa10b, DaLa11} and to the Stokes system in \cite{Da11}, and to the case of problems  in an infinite periodically perforated domain in \cite{LaMu12, Mu12a}.

We also mention that  boundary value problems in domains with periodic inclusions have been  analysed, at least for the two dimensional case, with the so-called method of functional equations (cf., {\it e.g.}, Castro and Pesetskaya \cite{CaPe10}, Castro, Pesetskaya, and Rogosin \cite{CaPeRo09},   Drygas and Mityushev \cite{DrMi09}, Mityushev \cite{Mi01}, Rogosin, Dubatovskaya, and Pesetskaya \cite{RoDuPe09}).
 
\section{Strategy} 

We briefly outline our strategy. First of all we recall that boundary value problem \eqref{bvpe}, which we consider only for positive $\epsilon$, is singular for $\epsilon=0$. Then, if $\epsilon$ is in $]0,\epsilon_0[$ we can convert problem \eqref{bvpe} into an equivalent system of integral equations defined on the $\epsilon$-dependent domain $\partial \Omega_{p,\epsilon}$ by exploiting periodic potential theory (cf., \textit{e.g.}, \cite{LaMu11}). Then, by an appropriate change of functional variables, we can desingularize the problem and obtain an equivalent system of integral equations defined on the fixed domain $\partial \Omega$. By means of the Implicit Function Theorem for real analytic maps in Banach spaces, we can analyse the dependence upon $\epsilon$ of the solutions of the system of integral equations and we can prove our main results. Further details will be presented in a forthcoming paper by the authors (see \cite{DaMu12a}).

\section{Main results}

\begin{theorem}
Let $(k,j) \in \{1,\dots,n\}^2$. Then there exist $\epsilon_1 \in ]0,\epsilon_0[$, an open neighbourhood $\mathcal{U}_{r_{\ast}}$ of $r_{\ast}$, and a real analytic function $\Lambda_{kj}$ from $]-\epsilon_1,\epsilon_1[\times{\mathcal{U}}_{r_{\ast}}$ to $\mathbb{R}$ such that $\epsilon/\rho(\epsilon) \in \mathcal{U}_{r_\ast}$ for all $\epsilon \in ]0,\epsilon_1[$ and such that 
\begin{equation}\label{eq:S0} 
\begin{split}
\lambda^{\mathrm{eff}}_{kj}[\epsilon]=&\lambda^-\delta_{k,j}+ \epsilon^n
\Lambda_{kj}\Bigl[\epsilon, \frac{\epsilon}{\rho(\epsilon)}\Bigr]\, ,
\end{split}
\end{equation}
for all $\epsilon \in ]0,\epsilon_1[$.
\end{theorem}

For a proof, we refer to \cite{DaMu12a}. Here, we note that if $\epsilon/\rho(\epsilon)$ has a real analytic continuation around $0$, then the term in the right hand side of equality \eqref{eq:S0} defines a real analytic function of the variable $\epsilon$ in the whole of a neighbourhood of $0$. Accordingly, the term in the left hand side of equality \eqref{eq:S0}, which is defined only for positive values of $\epsilon$, can be continued real analytically for $\epsilon \leq 0$. As a consequence, $\lambda^{\mathrm{eff}}_{kj}[\epsilon]$ can be expressed for $\epsilon$ small and positive in terms of a power series which converges absolutely on a whole neighbourhood of $0$. 

Moreover, we give in the following Theorem \ref{L0rast} more information on $\lambda^{\mathrm{eff}}_{kj}[\epsilon]$ for $\epsilon$ close to $0$  by expressing  $\Lambda_{kj}[0,r_\ast]$ by means of a certain quantity related to the solutions of a limiting transmission problem (for a proof we refer to \cite{DaMu12a}).

\begin{theorem}\label{L0rast}
Let $(k,j)\in\{1,\dots,n\}^2$. Then 
\[
\begin{split}
\Lambda_{kj}[0,r_{\ast}]=&\lambda^+\int_{\partial \Omega}  \tilde{u}_j^+ (t)(\nu_{\Omega}(t))_k\, d\sigma_t\\&-\lambda^-\int_{ \partial \Omega}  \tilde{u}_j^- (t)(\nu_{\Omega}(t))_k\, d\sigma_t\\
&+(\lambda^+-\lambda^-)|\Omega|_n\;\delta_{k,j}\,,
\end{split}
\]  where $|\Omega|_n$ denotes the $n$-dimensional measure of $\Omega$, and where $(\tilde{u}_j^+,\tilde{u}_j^-)$ is the unique solution in $C^{1}(\mathrm{cl}\Omega)\times C^{1}(\mathbb{R}^n\setminus\Omega)$ of the following transmission problem
\[
\left\{
\begin{array}{ll}
\Delta \tilde{u}_j^+=0 & {\mathrm{in}}\ \Omega\,,\\
\Delta \tilde{u}_j^-=0 & {\mathrm{in}}\ \mathbb{R}^n\setminus\mathrm{cl}\Omega\,,
\\
\lambda^-\frac{\partial \tilde{u}_j^-}{\partial\nu_{ \Omega}}(x)&\\ =\lambda^+\frac{\partial\tilde{u}_j^+}{\partial\nu_{ \Omega}}(x)+(\lambda^+-\lambda^-)(\nu_\Omega(x))_j& \forall x\in  \partial\Omega\,,\\
\lambda^+\frac{\partial\tilde{u}_j^+}{\partial\nu_{ \Omega}}(x)&\\   =r_{\ast}\Bigl(\tilde{u}_j^-(x)-\tilde{u}_j^+(x)\Bigr)-\lambda^+(\nu_{\Omega}(x))_j& \forall x\in  \partial\Omega\,,\\
\int_{\partial \Omega}\tilde{u}_j^+(x)\, d\sigma_x=0\, ,\\
\int_{\partial \Omega}\tilde{u}_j^-(x)\, d\sigma_x=0\, ,\\
\lim_{x\to \infty}\tilde{u}_j^-(x)\in\mathbb{R}\, .
\end{array}
\right.
\]
\end{theorem}

If we also assume that 
\[
r_\ast=\lim_{\epsilon\to 0^+} \frac{\epsilon}{\rho(\epsilon)}=0
\]
then the  expression for $\Lambda_{kj}[0,r_\ast]=\Lambda_{kj}[0,0]$ is simpler and we have the following.

\begin{proposition}\label{L00}
Let $r_\ast=0$. Let $(k,j)\in\{1,\dots,n\}^2$. Then 
\begin{equation}\label{eq:L2}
\begin{split}
\Lambda_{kj}[0,0]=&-\lambda^-\int_{ \partial \Omega}  \tilde{v}_j^{-} (t)(\nu_{\Omega}(t))_k\, d\sigma_t\\&-\lambda^-|\Omega|_n\delta_{k,j}\,,
\end{split}
\end{equation}  where $\tilde{v}_j^{-}$ is the unique solution in $C^{1}(\mathbb{R}^n\setminus\Omega)$ of the following  exterior Neumann problem
\begin{equation}\label{eq:lim4}
\left\{
\begin{array}{ll}
\Delta \tilde{v}_j^-=0 & {\mathrm{in}}\ \mathbb{R}^n\setminus\mathrm{cl}\Omega\,,
\\
\frac{\partial \tilde{v}_j^-}{\partial\nu_{ \Omega}}(x)=-(\nu_\Omega(x))_j& \forall x\in  \partial\Omega\,,\\
\int_{\partial \Omega}\tilde{v}_j^-(x)\, d\sigma_x=0\, ,\\
\lim_{x\to \infty}\tilde{v}_j^-(x)\in\mathbb{R}\, .
\end{array}
\right.
\end{equation}.
\end{proposition}
\begin{proof} Since $r_\ast=0$ and by Theorem \ref{L0rast} one deduces that 
\[
\begin{split}
\Lambda_{kj}[0,0]=&\lambda^+\int_{\partial \Omega}  \tilde{v}_j^+ (t)(\nu_{\Omega}(t))_k\, d\sigma_t\\&-\lambda^-\int_{ \partial \Omega}  \tilde{v}_j^- (t)(\nu_{\Omega}(t))_k\, d\sigma_t\\
&+(\lambda^+-\lambda^-)|\Omega|_n\delta_{k,j}\,,
\end{split}
\]  where $\tilde{v}_j^{-}$ is the unique solution in $C^{1}(\mathbb{R}^n\setminus\Omega)$  of \eqref{eq:lim4} and where $\tilde{v}_j^+$ is the unique solution in $C^{1}(\mathrm{cl}\Omega)$ of
\[
\left\{
\begin{array}{ll}
\Delta \tilde{v}_j^+=0 & {\mathrm{in}}\ \Omega\,,\\
\frac{\partial\tilde{v}_j^+}{\partial\nu_{ \Omega}}(x)=-(\nu_{\Omega}(x))_j& \forall x\in  \partial\Omega\,,\\
\int_{\partial \Omega}\tilde{v}_j^+(x)\, d\sigma_x=0\, .
\end{array}
\right.
\] It follows that 
\[
\tilde{v}_j^+(x)=-x_j+\frac{1}{|\partial\Omega|_{n-1}}\int_{\partial\Omega}t_j\,d\sigma_{t} \quad \forall x \in \mathrm{cl}\Omega\,.
\] Then the divergence theorem implies that  
\[
\int_{\partial \Omega}  \tilde{v}_j^+ (t)(\nu_{\Omega}(t))_k\, d\sigma_t=-|\Omega|_n\delta_{k,j}\,.
\] Now the validity of the proposition follows by a straightforward calculation. \end{proof}

\medskip

If we further assume that $\Omega$ is the unit ball $\mathbb{B}_n$ in $\mathbb{R}^n$, then we have the following.

\begin{proposition}\label{L00ball}
Let $r_\ast=0$. Assume that $\Omega=\mathbb{B}_n$. Let $(k,j)\in\{1,\dots,n\}^2$.  Then 
\[
\Lambda_{kj}[0,0]=-\lambda^-\frac{ s_{n}}{n-1}\; \delta_{k,j}
\] where $s_n$ denotes the $(n-1)$-dimensional measure of $\partial\mathbb{B}_n$.
\end{proposition}
\begin{proof} By assumption $\Omega=\mathbb{B}_n$, one verifies that the unique solution of problem \eqref{eq:lim4} is given by
\[
\tilde{v}_j^-(x)\equiv \frac{1}{n-1}\; \frac{x_j}{|x|^n}\qquad\forall x\in\mathbb{R}^n\setminus\mathbb{B}_n\,.
\] Then, by the divergence theorem one has
\[
\int_{\partial \mathbb{B}_n}  \tilde{v}_j^- (t)(\nu_{\mathbb{B}_n}(t))_k\, d\sigma_t=\frac{1}{n-1}|\mathbb{B}_n|_n\;\delta_{k,j}
\] where $|\mathbb{B}_n|_n$ denotes the $n$-dimensional measure of $\mathbb{B}_n$. Now the validity of the proposition follows by equality \eqref{eq:L2} and by a straightforward calculation (also note that $s_n=n|\mathbb{B}_n|_n$). \end{proof}

\section{Further remarks}

We observe that we can investigate also the asymptotic behaviour of suitable restrictions of the functions $u^+_j[\epsilon]$ and $u^-_j[\epsilon]$ as $\epsilon$ tends to $0$. Moreoveor, we can analyse also the case in which we add to the fifth and sixth conditions in \eqref{bvpe} suitable functions defined on $\partial \Omega_{p,\epsilon}$, and thus we consider non-homogeneous boundary conditions.

\section*{Acknowledgments}
The first named author acknowledges financial support from the Foundation for Science and Technology (FCT) via the post-doctoral grant SFRH/BPD/64437/2009. The work of the first named author was also supported by {\it FEDER} funds through {\it COMPETE}--Operational Programme Factors of Competitiveness (``Programa Operacional Factores de Competitividade'') and by Portuguese funds through the {\it Center for Research and Development in Mathematics and Applications} (University of Aveiro) and the Portuguese Foundation for Science and Technology (``FCT--Funda\c{c}\~{a}o para a Ci\^{e}ncia e a Tecnologia''), within project PEst-C/MAT/UI4106/2011 with COMPETE number FCOMP-01-0124-FEDER-022690.  The second named author acknowledges the financial support of the ``Fondazione Ing.~Aldo Gini''.

\end{document}